\documentclass{article}
\usepackage{amsmath,amssymb,bbm,color,amsthm}
\usepackage{graphicx}
\usepackage{enumerate,enumitem}

\usepackage{cleveref}

\graphicspath{{Pics/}}

\newtheorem{theorem}{Theorem}

\newtheorem{lemma}[theorem]{Lemma}

\newtheorem{corollary}[theorem]{Corollary}

\crefname{lemma}{Lemma}{Lemmas}
\crefname{theorem}{Theorem}{Theorems}

\newtheoremstyle{definition}% name
 {4pt}% Space above
 {4pt}% Space below
 {\sl}% Body font
 {}% Indent amount (empty = no indent, \parindent = para indent)
 {\bfseries}% Thm head font
 {.}% Punctuation after thm head
 {.5em}% Space after thm head: " " = normal interword space;
 {}% Thm head spec (can be left empty, meaning `normal')
\theoremstyle{definition}

\theoremstyle{remark}

\theoremstyle{plain}
\newtheorem*{claim*}{Claim}

\newtheoremstyle{introthms}% name\Gr
 {3pt}% Space above, empty = `usual value'
 {3pt}% Space below
 {\itshape}% Body font
 {}% Indent amount (empty = no indent, \parindent = para indent)
 {\bfseries}% Thm head font
 {.}% Punctuation after thm head
 {.5em}% Space after thm head: " " = normal interword space;
 {\thmnote{#3}}% Thm head spec
\theoremstyle{introthms}
\newcommand{\cF}{\mathcal{F}}
\newcommand{\Tr}{{\rm Tr}}
\newcommand{\Gr}{{\rm Gr}}
\newcommand{\UOP}{{\rm UOP}}

\begin{document}

\title{Planar Ramsey graphs}
\author{ M. Axenovich\thanks{Department of Mathematics, Karlsruhe Institute of Technology}, C. Thomassen\thanks{Department of Applied Mathematics and Computer Science,
Technical University of Denmark}, U. Schade\thanks{Department of Mathematics, Karlsruhe Institute of Technology}, T. Ueckerdt\thanks{Computer Science Department, Karlsruhe Institute of Technology}}

\maketitle
 
\begin{abstract}
 We say that a graph $H$ is {\it planar unavoidable} if there is a planar graph $G$ such that any red/blue coloring of the edges of $G$ contains a monochromatic copy of $H$, otherwise we say that $H$ is {\it planar avoidable}. 
 I.e., $H$ is planar unavoidable if there is a Ramsey graph for $H$ that is planar. 
 It follows from the Four-Color Theorem and a result of Gon\c{c}alves that if a graph is planar unavoidable then it is  bipartite and outerplanar.
 We prove that the cycle on $4$ vertices and any path are planar unavoidable. 
 In addition, we prove that all trees of radius at most~$2$ are planar unavoidable and there are trees of radius~$3$ that are planar avoidable. 
 We also address the planar unavoidable notion in more than two colors.
\end{abstract}

\section{Introduction}

Ramsey's theorem~\cite{R} claims that any graph is Ramsey in the class of all complete graphs, i.e., for any graph $G$ and any number $k$ of colors there is a sufficiently large complete graph such that in any coloring of its edges in $k$ colors there is a monochromatic copy of $G$. 
In general for graphs $G$ and $H$, we write $G\rightarrow_k H$ and say that $G$ $k$-{\it arrows} $H$ if any coloring of the edges of $G$ in $k$ colors contains a monochromatic copy of $H$. 
We write $G\rightarrow H$ and say that $G$ {\it arrows}  $H$ if $k=2$.
There are classes of graphs that are Ramsey in their own class, meaning that for any graph $H$ in a class $\cF$ there is a graph $G\in \cF$ such that $G\rightarrow H$. 
Examples of such classes include bipartite graphs, graphs with a given clique number, and graphs of a given odd girth, see~\cite{NR1, NR2}.
Here, we are concerned with Ramsey properties of the class of all planar graphs.
We say that a planar graph $H$ is $k$-{\it planar unavoidable} if there is a planar graph $G$ such that $G\rightarrow_k H$, otherwise we call $H$ $k$-{\it planar avoidable}. Similarly, we define {\it outerplanar unavoidable} and {\it outerplanar avoidable} graphs.
When $k=2$, we write {\it planar unavoidable} instead of $2$-{\it planar unavoidable}, or, if clear from context, simply {\it unavoidable}.
The complexity of the problem to edge-color planar graphs with a given number of colors so that there is no monochromatic copy of a given graph was addressed by Broersma et al.~\cite{BFKW}. A related problem of bounding local density of Ramsey graphs has been addressed for example  in \cite{RR} and \cite{MP}. \\

A result of Gon\c{c}alves~\cite{Go} states that any planar graph can be edge-colored in two colors so that each color class is an outerplanar graph. 
Thus any planar unavoidable graph is outerplanar. 
The Four Color Theorem~\cite{AH} implies that any planar graph is a union of two bipartite graphs. 
In general any graph that is $2^k$-colorable for $k\in \mathbb{N}$ is a union of at most $k$ bipartite graphs.\\

This shows that any planar unavoidable graph is bipartite and outerplanar and thus gives necessary conditions for planar unavoidability.\\

Next we give several sufficient conditions. 
Here, a {\it generalized broom} is a union of a path and a star such that they share only the center of the star.

\begin{theorem}\label{good2}
If $H$ be a path, a cycle on $4$ vertices, a tree of radius at most $2$, or a generalized broom, then $H$ is planar unavoidable.
 Moreover, if $H$ is a path, then it is outerplanar unavoidable.
\end{theorem}

The next result shows that not only odd cycles and non-outerplanar graphs are planar avoidable, but also some trees. 

\begin{theorem}\label{not-good}
 There is a planar avoidable tree of radius $3$ and an outerplanar avoidable tree of radius $2$.
\end{theorem}

Moreover, any planar avoidable tree has at least $8$ vertices and there is a planar avoidable tree on $106$ vertices. 

A result of Hakimi et al.~\cite{Ha}, see also~\cite{AA}, states that any planar graph can be edge-decomposed into at most five star forests. 
Thus the $k$-planar unavoidable graphs for $k\geq 5$ are precisely  the star forests.
Next we summarise our results for $k$-planar unavoidable graphs, for $k=3$ and $4$.

\begin{theorem}\label{more-colors}
 If $H$ is $k$-planar unavoidable for $k\geq 3$, then $H$ is a forest. 
 If $H$ is $4$-planar unavoidable, then $H$ is a caterpillar forest. 
 There are $3$- and $4$-planar avoidable trees of radius~$2$. 
\end{theorem}

Moreover,  there are $3$- and $4$-planar avoidable trees on $10$ and $6$ vertices, respectively.

We provide some definitions in Section~\ref{definitions}. 
Sections~\ref{sec:good-2},~\ref{sec:not-good-1}, and~\ref{sec:more-colors} contain the proofs of Theorems~\ref{good2},~\ref{not-good}, and~\ref{more-colors} respectively.
Finally Section~\ref{conclusions} states some concluding remarks and open questions.

\section{Definitions}\label{definitions}

We denote a complete graph, a path, and a cycle on $n$ vertices by $K_n, P_n,$ and $ C_n$, respectively. 
A complete bipartite graph with parts of sizes $m$ and $n$ is denoted by $K_{m,n}$.
For an integer $k$, $k\geq 2$, a {\it $k$-ary tree} is a rooted tree in which each vertex has at most $k$ children. 
A {\it perfect} $k$-ary tree is a $k$-ary tree in which every non-leaf vertex has $k$ children and all leaf vertices have the same distance from the root.
For all other standard graph theoretic definitions, we refer the reader to the book of West~\cite{W}.\\

\noindent
{\bf Iterated Triangulation $\Tr(n)$}:\\
An {\it iterated triangulation} is a plane graph $\Tr(n)$ defined as follows: $\Tr(0)=K_3$ is a triangle, $\Tr(i) \subseteq \Tr(i+1)$, $\Tr(i+1)$ is obtained from $\Tr(i)$ by inserting a vertex in each of the inner faces of $\Tr(i)$ and connecting this vertex with edges to all the vertices on the boundary of the respective face, see Figure~\ref{Def Tr(n)}.
We see that $\Tr(i)$ is a  triangulation and each triangle of $\Tr(i)$ bounds a face of $\Tr(j)$ for some $j \leq i$.\\

\begin{figure}[htpb]
 \centering
 \includegraphics{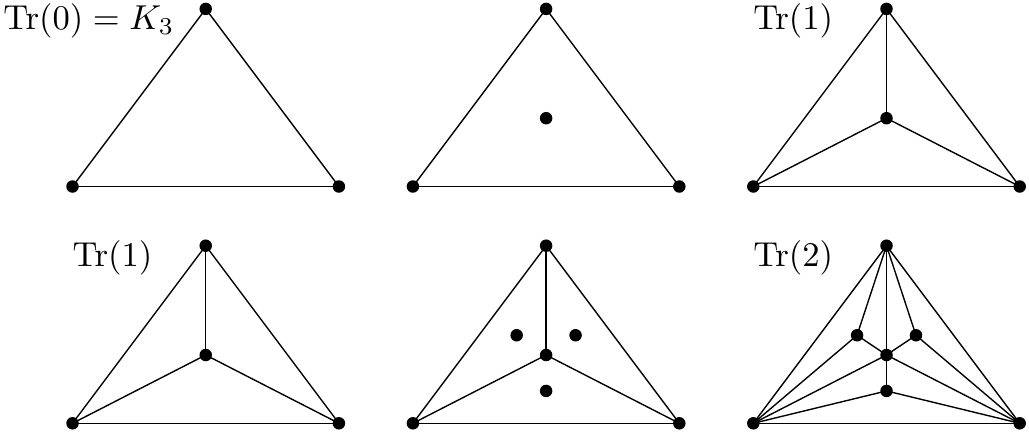}
 \caption{The iterative construction of $\Tr(2)$.}
 \label{Def Tr(n)}
\end{figure}

\noindent
{\bf Universal outerplanar graph $\UOP(n)$:}\\
A {\it universal outerplanar graph} $\UOP(n)$ is defined as follows: $\UOP(1)$ is a triangle. An edge on the outer face is called an {\it outer} edge.
For $k>1$, $\UOP(k)$ is an outerplanar graph that is a supergraph of $\UOP(k-1)$ obtained by introducing, for each outer edge $e=xy$, a new vertex $v_e$ and new edges: $v_e x$ and $v_ey$. Then the set of outeredges of $\UOP(k)$ is $\{v_ex, v_ey: ~ e=xy \mbox{ is an outeredge of } \UOP(k-1)\}.$\\

\noindent
{\bf Triangulated Grid $\Gr (n)$:}\\
Let a {\it triangulated grid} be a graph $\Gr(n)=(V,E)$ with $V=[n]\times[n]$ and $(k,j)(k',j')\in E$ if and only if either ($k=k'$ and $|j-j'|=1$) or ($|k-k'|=1$ and $j=j'$) or ($k=k'-1$ and $j=j'-1$) or ($k=k'+1$ and $j=j'+1$).
We define left, right, top, and bottom sides of the grid as subsets of vertices 
$[n]\times \{1\}$, $[n]\times \{n\}$, $\{1\}\times [n]$, and $\{n\}\times [n]$ respectively.\\

\noindent
{\bf Fish:}\\
A graph $G$ is called a {\it fish} and denoted $F_{x,y}$ if $V(G)=\{x, y\}\cup S$, where $S\cap \{x,y\} = \emptyset$, $x$ and $y$ are each adjacent to each vertex in $S$, $S$ induces a path in $G$, and $xy$ is an edge.
We call $S$ the set of {\it spine vertices}, $G[S]$ is called the {\it spine}, $xs, ys$ are called {\it ribs}, $s\in S$, and the paths $x,s,y$ of length~$2$ are called {\it double ribs}.
Sometimes we say that a fish $F_{x,y}$ {\it hangs} on an edge $xy$.
In an edge-colored fish, a double rib is called {\it bicolored} if there are different colors used on two edges of this double rib.
We will call two double ribs $x,s,y$ and $x,s',y$, with $s\neq s'$, $s,s'\in S$, {\it identically bicolored}, if the same color is used on both of the edges $xs$ and $xs'$, and a different color is used on both of the edges $sy$ and $s'y$.
Note that for any positive integers $m$ and $k$ and for any edge $xy \in E(\Tr(m))$, there is a fish on $xy$ in $\Tr(m+k)$ with $k$ spine vertices.
Indeed, consider an inner face $xyz$ in $\Tr(m)$.
We can pick spine vertices $s_1,\ldots,s_k$ such that $s_i\in V(\Tr(m+i)-\Tr(m+i-1))$, $i\in\{1,\dots,k\}$ and such that $s_i$ is inserted in the face $xys_{i-1}$ of $\Tr(m+i-1)$, $i=1,\ldots,k$, $s_0=z$.\\

\begin{figure}[h]
 \centering
 \includegraphics{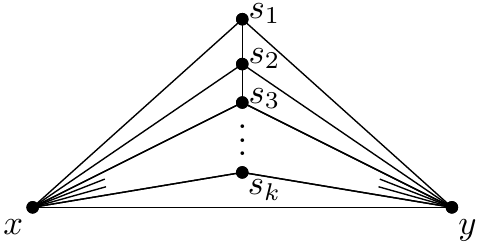}
 \caption{A fish $F_{x,y}$ with $k$ spine vertices.}
 \label{DefFish}
\end{figure}

%%%%%%%%%%%%%%%%%%%%%%%%%%%%%%%%
%%%%%%%%%%%%%%%%%%%%%%%%%%%%%%%%
%%%%%%%%%%%%%%%%%%%%%%%%%%%%%%%%
\section{Proof of Theorem~\ref{good2}}\label{sec:good-2}
%%%%%%%%%%%%%%%%%%%%%%%%%%%%%%%%
%%%%%%%%%%%%%%%%%%%%%%%%%%%%%%%%
%%%%%%%%%%%%%%%%%%%%%%%%%%%%%%%%
\noindent
Theorem~\ref{good2} follows immediately from the following lemmas.

The following proof closely resembles the Hex-lemma~\cite{Ga}. 

\begin{lemma}\label{path}
 Let $G$ be a near-triangulation with outer cycle $C$, that is, $G$ is a planar graph with outer face boundary $C$ and each other face is bounded by a triangle.
 Let $a,b,c,d$ be vertices on $C$ in clockwise order dividing the edges of $C$ in four paths $C(a,b),C(b,c),C(c,d),C(d,a)$, respectively.
 If the edges of $G$ are colored red and blue, then either there is a blue path from $C(a,b)$ to $C(c,d)$ or a red path from $C(b,c)$ to $C(d,a)$ (or both).
\end{lemma}
\begin{proof}
 Suppose there is no blue path from $C(a,b)$ to $C(c,d)$.
 Then the red graph contains a minimal edge-cut separating $C(a,b)$ and $C(c,d)$.
 A minimal edge-cut in $G$ is a cycle $C'$ in the dual graph $G^*$.
 This cycle $C'$ must contain the vertex $v^*$ corresponding to the outer face of $G$.
 Since $G$ is a near-triangulation, it follows that any two consecutive edges of $C'$ (except the two edges incident with $v^*$) correspond to two edges in $G$ that are incident with the same vertex.
 Thus the edges in $G$ corresponding to the edges in $C'$ contain a red path path from $C(b,c)$ to $C(d,a)$.
\end{proof}

\begin{corollary}
 Any path is planar unavoidable, even in a class of planar graphs of bounded degrees (in fact of maximum degree at most $6$).
\end{corollary}
\begin{proof}
 If the edges of the triangulated grid $\Gr(k)$ are colored red or blue, then there is a monochromatic $P_k$ by Lemma~\ref{path}, where the paths $C(a,b),C(b,c),C(c,d),C(d,a)$ correspond to the top, right, bottom, and the left sides of the grid.
\end{proof}

The above gives planar graphs of bounded maximum degree that arrow arbitrarily long paths, which however have large tree-width.
Complementary, we can find planar graphs of tree-width~$2$ that also arrow arbitrarily long paths, where however the maximum degree is large.

\begin{lemma}\label{lem:path}
 Any path is outerplanar unavoidable.
 In particular, for any positive integer $n$, $\UOP(n^2) \rightarrow P_n$.
\end{lemma}
\begin{proof}
 We shall show that $\UOP(n^2)\rightarrow P_n$.
 Let $G = G_{n^2} = \UOP(n^2)$ and let it be edge-colored red and blue.
 We see that each edge of $G$ is on the outer face of $G_i=\UOP(i)$ for some $i \leq n^2$, where $G_1\subseteq G_2\subseteq \cdots \subseteq G_{n^2}$ as in the definition of the universal outer planar graph. 
 Consider the unique outerplanar embedding of $G$ and for each edge $e$, consider $G_i$ such that $e$ is on the outer face of $G_i$.
 For a vertex let its rank be the least $i \in \{1,\ldots,n^2\}$ for which it is in the vertex set of $G_i$.
 For each edge $e$, we define graphs $G({\rm out},e)$ and $G({\rm in}, e)$ such that $G = G({\rm out}, e) \cup G({\rm in}, e)$, where $G({\rm out}, e)$ and $G({\rm in}, e)$ share only the edge $e$ and no vertices except for the endvertices of $e$.
 We require in addition that $G({\rm in}, e)$ contains $G_1$ as a subgraph, see Figure~\ref{fig:outerplanar-definitions}.
 Observe that among vertices of rank $i$ in $G$ there are two at distance $i$ in $G$, $i=1,\ldots,n^2$.
 
 \begin{figure}[t]
  \centering
  \includegraphics{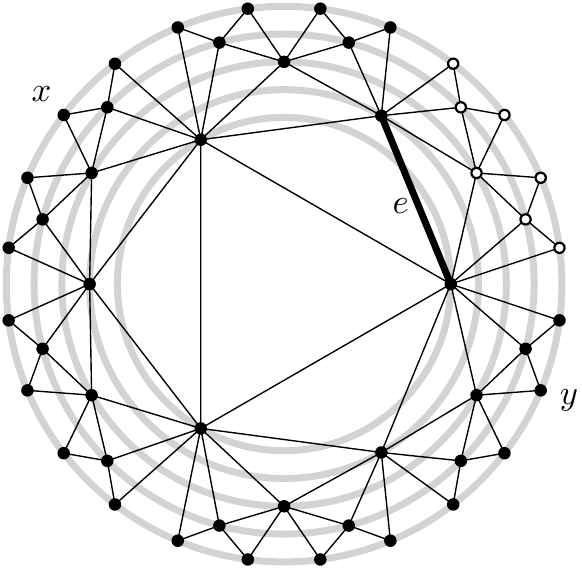}
  \caption{The universal outerplanar graph $\UOP(5)$.
   Vertices with the same rank lie on concentric circles.
   For the thick edge $e$, the vertices in $G({\rm in},e)$ are shown in black.
   The two vertices $x,y$ in $\UOP(5)$ have distance~$5$.}
  \label{fig:outerplanar-definitions}
 \end{figure}
 
 For an edge $e$ in $G_i$ with endvertex $v$ of rank $i$, we define the following.
 Let $R(e)$ be a longest red path in $G_i$ with last edge $e$ and last vertex $v$.
 Let $B(e)$ be a longest blue path in $G({\rm in},e)$ with last vertex $v$.
 We shall write $e>e'$ for two edges of $G$ if $|R(e)| > |R(e')|$, or $|R(e)| = |R(e')|$ and $|B(e)| > |B(e')|$.\\
 
 Consider the edges in $G_n$.
 Assume that the endvertices of each such edge belong to the same blue component.
 Then for any two vertices of rank $n$, there is a blue path joining them. 
 Since there are two such vertices at distance at least $n$ in $G$, we see that there is a blue path on $n$ edges, and we are done. 
 So, assume that $e_n$ is an outer edge of $G_n$ such that its endvertices belong to different blue components of $G$. 
 Assume we constructed a sequence of edges $e_n < e_{n+1}< \cdots < e_{n+i}$ of outer edges in $G_n,\ldots,G_{n+i}$ respectively such that the endvertices of each of these edges belong to different blue components of $G$.
 Consider $e=e_{n+i}$, we shall construct $e_{n+i+1}$.
 Let $e=uv$ with $v$ of rank $n+i$ and let $e', e''$ be two adjacent outer edges of $G_{n+i+1}$ that are incident to $u$ and $v$, respectively.
 Let $e'=uw$ and $e''=vw$ where $w$ has rank $n+i+1$.
 Note that either ($u$ and $w$) or ($v$ and $w$) are in different blue components in $G$, otherwise $u$ and $v$ would have been in the same blue component.
 See Figure~\ref{fig:outerplanar-cases} for illustrations.\\
 
 \begin{figure}[t]
  \centering
  \includegraphics{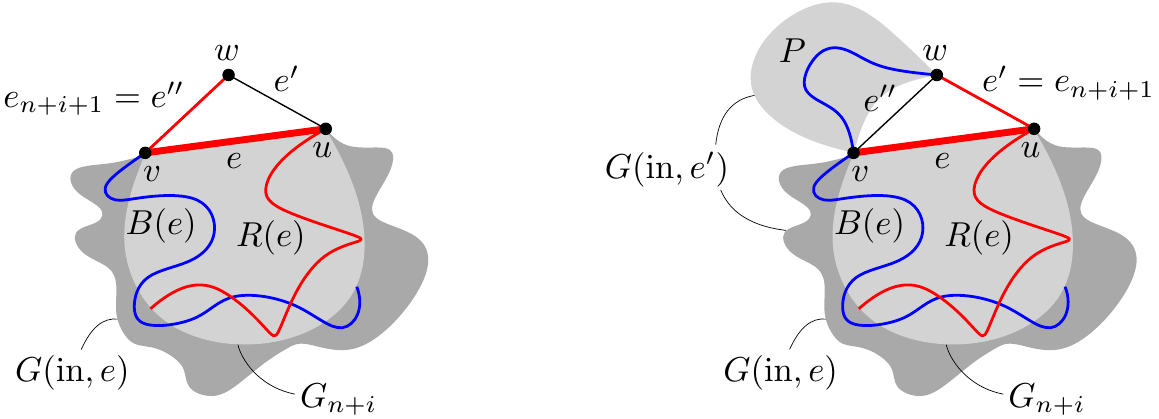}
  \caption{Illustrations of Case 1 (left) and Case 2 (right) in the proof of Lemma~\ref{lem:path}.}
  \label{fig:outerplanar-cases}
 \end{figure} 
 
 Case 1. $v$ and $w$ are in different blue components.
 Then $e'' = vw$ is red and the path $R(e) \cup vw$ is a red path in $G_{n+i+1}$ of length $|R(e)|+1$ ending in $e''$ at vertex $w$.
 Then let $e_{n+i+1} = e''$.
 We see that $e'' > e$. \\
 
 Case 2. $v$ and $w$ are in the same blue component and $u$ and $w$ are in different blue components.
 Then $e' = uw$ is red and the path $(R(e) - uv) \cup uw$ is a red path of length $|R(e)|$ in $G_{n+i+1}$ ending with $e'$ at vertex $w$.
 Since $v$ and $w$ are in the same blue component, there is a blue path $P$ of length $q$, $q\geq 1$, between them in $G({\rm out}, e'')$.
 The union of $P$ and the blue path $B(e)$ ending at $v$ in $G({\rm in}, e)$ forms a blue path ending at $w$ in $G({\rm in}, e')$.
 Let $e_{n+i+1}= e'$.
 We see that $e' > e$.\\
  
 We can continue in this manner until rank $n^2$, i.e., we create a desired sequence $e_n < \cdots < e_{n^2}$. 
 Note that $|R(e_i)| \geq 1$ and $|B(e_i)| \geq 0$ for $i = n,\ldots,n^2$, and $(|R(e_i)|,|B(e_i)|) \neq (|R(e_j)|,|B(e_j)|)$ whenever $i \neq j$.
 As there are exactly $n^2-n+1 = (n-1)n + 1$ edges in this sequence, there exists some $i \in \{n,\ldots,n^2\}$ with $|R(e_i)| \geq n$ or $|B(e_i)| \geq n$, proving that there is a red or a blue path of length at least $n$.
\end{proof}

\begin{lemma}\label{Broomhd}
 Any generalized broom is planar unavoidable. 
\end{lemma}
\begin{proof}
 Let $H$ be a union of $P_{2k+1}$ and $K_{1,k}$ that share only their center vertices.
 Note that any generalized broom on at most $k$ vertices is a subgraph of $H$.
 Let $n$ be sufficiently large, say $n\geq 10k^2$.
 Consider $G=\Tr(10n)$ colored red and blue.
 Since $\UOP(8n)\subseteq \Tr(8n)$, we see that there is monochromatic path $P$ on edges $e_1, \ldots, e_n$ in order in a two-edge colored $\Tr(8n)$, say $P$ is red. 
 Consider a set $\cF$ of $n-2k$ fishes hanging on $e_{k+1}, e_{k+2}, \ldots, e_{n-k}$ respectively such that the spines of fishes from $\cF$ are pairwise disjoint and each fish has at least $4k$ spine vertices.
 If at least one of these fishes contains a red star of size $k$ centered at a vertex of $P$, we have a red $H$.
 Otherwise, each fish in $\cF$ contains at least $2k$ blue double ribs.
 The union of blue subgraphs of fishes from $\cF$ clearly contains a blue copy of $H$.
\end{proof}

\begin{lemma}\label{lem:height-2}
 Any tree of radius $2$ is planar unavoidable.
\end{lemma}
\begin{proof}
 Let $H$ be a perfect $k$-ary tree of radius $2$. 
 Consider $\Tr(19k)$ together with a fixed edge coloring in red and blue. 
 
 \begin{figure}[tb]
  \centering
  \includegraphics{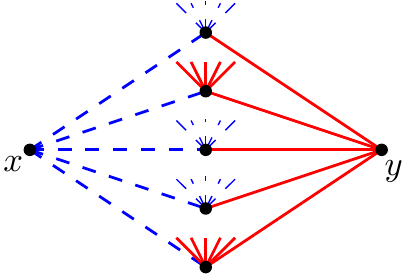}
  \caption{Two monochromatic stars of different colors with the same leaf-set.}
  \label{TCase1}
 \end{figure}

 \begin{claim*}
  If there is a red star $S_r$ and a blue star $S_b$ on $2k$ edges in $\Tr(n)$, $n \leq 18k$, such that the stars have the same leaf-set $L$, then there is a monochromatic copy of $H$ in $\Tr(n+k)$.
 \end{claim*}

 Let $x,y$ denote the centers of $S_r$ and $S_b$, respectively.
 Each vertex $z \in L$ has at least $2k$ neighbors in $\Tr(n+k)$ that are not neighbors of any vertex in $L - z$.
 Hence, by pigeonhole principle $z$ is the center of a monochromatic star on $k$ edges, whose leaves have distance at least two to $L - z$, see Figure~\ref{TCase1}.
 At least $k$ of these monochromatic stars are of the same color that together with either $S_r$ or $S_b$ form a monochromatic copy of $H$.
 This proves the Claim.

 \begin{figure}[tb]
  \centering
  \includegraphics[scale=0.75]{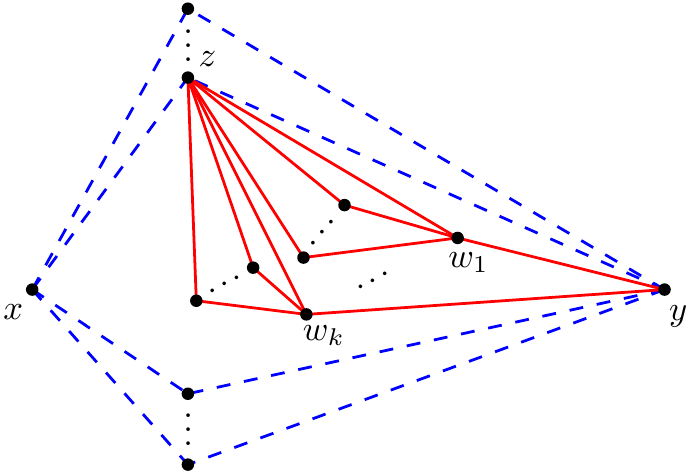}
  \caption{Part of a fish $F$ with $k$ blue double ribs, part of a fish $F_{x,s}$ with red double ribs, and part of a fish $F_{s,s'}$ between $s$ and $s'$ with monochromatic red double ribs.}
  \label{TCase2}
 \end{figure}

 Now consider any two adjacent vertices $x,y$ in $\Tr(n)$, $n \leq 12k$, and the set $L$ of their at least $6k$ common neighbors in $\Tr(n+6k)$.
 By the Claim, we may assume that fewer than $2k$ vertices of $L$ have a red edge to $x$ and a blue edge to $y$, and fewer than $2k$ vertices of $L$ have a blue edge to $x$ and a red edge to $y$.
 Each of the remaining at least $2k$ vertices in $L$ has its edges to $x$ and $y$ in the same color, and by pigeonhole principle we may assume that for at least $k$ of these vertices this the same color.
 We let $K(x,y)$ denote this monochromatic copy of $K_{2,k}$ in $\Tr(n+6k)$.

 Finally, consider two adjacent vertices $x,y$ in $\Tr(0)$.
 Say that $K(x,y) \subset \Tr(6k)$ is blue.
 If for every vertex $z$ in $K(x,y) - \{x,y\}$ we find a monochromatic $K(z,a) \subset \Tr(18k)$ in blue for some $a$, then there is a blue copy of $H$, as desired.
 So assume that for at least one vertex $z$ in $K(x,y)- \{x,y\}$ all monochromatic $K(z,a)$ for some $a$ are red; see Figure~\ref{TCase2}.
 Then in particular $K(z,y) \subseteq \Tr(12k)$ is red with vertices $z,y$ and $w_1,\ldots,w_k$.
 Moreover, for each $i=1,\ldots,k$ the monochromatic $K(z,w_i) \subset \Tr(18k)$ is red.
 However, this gives a red copy of $H$ rooted at $y$; see Figure~\ref{TCase2}.
\end{proof}

\begin{lemma}\label{C4}
 A cycle $C_4$ is planar unavoidable.
 For $n \geq 16$, $\Tr(n) \rightarrow C_4$.
\end{lemma}
\begin{proof}
 Consider the graph $G$ consisting of a fish $F_{x,y}$ hanging on edge $xy$ with $15$ spine vertices $s_1,\ldots,s_{15}$, and a vertex of degree three in each face of $F_{x,y}$ bounded by two spine vertices; see the left part of Figure~\ref{fig:C4-full-graph}.
 Note that $G \subset \Tr(15)$.
 Consider any fixed edge-coloring of $G$ in red and blue.
 
 \begin{figure}[tb]
  \centering
  \includegraphics{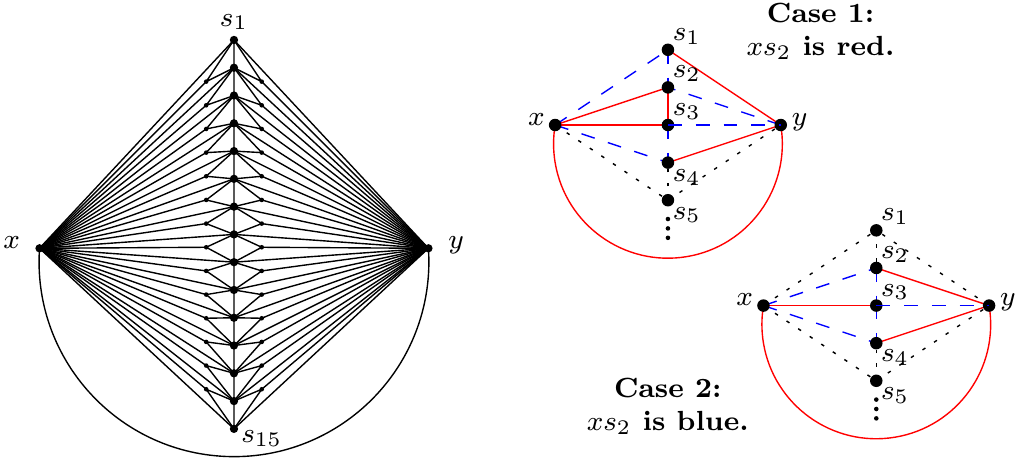}
  \caption{Left: A planar graph $G$ with $G \rightarrow C_4$.
   Right: Illustrations for the two cases in the proof of Lemma~\ref{C4}.}
  \label{fig:C4-full-graph}
 \end{figure}
 
 First we claim that $F_{x,y}$ contains a monochromatic $C_4$ or a monochromatic inner face $f$ such that any two vertices $u,v$ of $f$ have a common neighbor $w$ in $F_{x,y}$, not in $f$, such that edges $uw$ and $vw$ have the same color.
 To this end, consider the spine vertices $s_1,\ldots,s_{15}$ and the corresponding double ribs $x,s_i,y$, $i = 1,\ldots,15$.
 If $F_{x,y}$ contains no monochromatic $C_4$, at most two double ribs are monochromatic -- one red and one blue.
 Hence there are five consecutive spine vertices $s_i,\ldots,s_{i+4}$ whose double ribs are bicolored. Assume, without loss of generality, that $i=1$. Further assume that the edges $xy$ and $xs_{3}$ are red, so the edge $s_{3}y$ is blue.
 
 \begin{description}
  \item[Case 1: $xs_{2}$ is red.]
   Then $s_{2}y$ is blue.
   If the spine edge $s_{2}s_{3}$ is blue, then $s_{2},s_{3},y$ bound an inner face $f$ with the desired properties ensured by the vertex $x$ that sends red edges to $f$.
   So we may assume that $s_{2}s_{3}$ is red.
   For the same reason, if $xs_1$ is also red, then also $s_1s_{2}$ is red, giving a red $C_4$ with vertices $x,s_1,s_{2},s_{3}$.
   So we may assume that $xs_1$ is blue and hence $s_1y$ is red.
   Now if $s_1s_{2}$ is red, there is a red $C_4$ with vertices $x,y,s_1,s_{2}$.
   So we may assume that $s_1s_{2}$ is blue.
 
   Symmetrically, we may assume that $xs_{4}$ is blue, hence $s_{4}y$ is red, and $s_{3}s_{4}$ is blue.
   But now $s_{2},s_{3},x$ bound an inner face with the desired properties; see the right part of Figure~\ref{fig:C4-full-graph}.
 
  \item[Case 2: $xs_{2}$ is blue.]
   Then $s_{2}y$ is red.
   Now if $s_{2}s_{3}$ is red, we have a red $C_4$ with vertices $x,y,s_{2},s_{3}$.
   So we may assume that $s_{2}s_{3}$ is blue.
   By symmetry we may also assume that $xs_{4}$ is blue, hence $s_{4}y$ is red, and $s_{3}s_{4}$ is blue.
   But then we have a blue $C_4$ with vertices $x,s_{2},s_{3},s_{4}$; see the right part of Figure~\ref{fig:C4-full-graph}.
 \end{description}

 This proves the claim that $F_{x,y}$ contains a monochromatic $C_4$ or a monochromatic inner face $f$, say in red, such that any two vertices of $f$ are joined by a blue $P_3$ in $F_{x,y}$.
 In the former case we are done.
 In the latter case note that as $f$ is all red, any two vertices of $f$ are also joined by a red $P_3$ in $F_{x,y}$.
 Now consider the vertex $z$ in $G - F_{x,y}$ whose three neighbors are the vertices of $f$.
 As two of the three edges incident to $z$ have the same color, there are two vertices in $f$ that are joined by two distinct but identically colored $P_3$'s in $G$.
 That is, there is a monochromatic copy of $C_4$ in $G$.
\end{proof}

%%%%%%%%%%%%%%%%%%%%%%%%%%%%%%%%
%%%%%%%%%%%%%%%%%%%%%%%%%%%%%%%%
%%%%%%%%%%%%%%%%%%%%%%%%%%%%%%%%

\section{Proof of Theorem~\ref{not-good}}\label{sec:not-good-1}

Let $T_1$ be a tree of radius~$3$ with root $r$ and all vertices of distance $0,1,2$ to $r$ having degree~$5$. 
See Figure~\ref{fig:trees-and-colorings}.
Let $G$ be a planar graph.
Let $V_1, V_2, V_3$ be a partition of $V(G)$ such that each $V_i$ induces a linear forest in $G$, $i=1,2,3$, such a partition exists by a result of Poh~\cite{Poh}.
Further, consider an orientation of $G$ with out-degree at most~$3$ at each vertex, see~\cite{H}.
(This orientation result also follows from~\cite{NW}.)
For $i = 1,2,3$ color the edges in $G[V_i]$ alternately red and blue along the paths in $G[V_i]$.
For each remaining directed edge $uv$ of $G$ we have $u \in V_i$ and $v \in V_j$ for some $i \neq j$.
Color $uv$ red if $i < j$ and blue if $i > j$.
See Figure~\ref{fig:trees-and-colorings}.
Assume that there is a monochromatic copy of $T_1$, say red.
Since the out-degree of each vertex in $G$ is at most~$3$, we see that each non-leaf vertex of $T_1$ has at least two incoming edges.
Due to the color alternation in each $G[V_i]$, at least one of the two incoming edges has its two endvertices in distinct parts.
In particular, the root $r$ is in $V_2$ or in $V_3$.
Then at least one vertex at distance $1$ or $2$ from $r$ is in $V_1$.
However, the vertices of $V_1$ have in-degree at most~$1$ in the red graph, a contradiction.\\
 
Let $T_2$ be a tree of radius~$2$ with root $r$ and all vertices of distance $0,1$ to $r$ having degree~$4$.
See Figure~\ref{fig:trees-and-colorings}.
Similarly, let $G$ be an outerplanar graph.
Let $V_1, V_2$ be a partition of $V(G)$ such that each $V_i$ induces a linear forest in $G$, $i=1,2$, such a partition exists by a result of Cowen et al.~\cite{CCW}.
Further, consider an orientation of $G$ with out-degree at most~$2$ at each vertex, see~\cite{H,NW}.
For $i = 1,2$ color the edges in $G[V_i]$ alternately red and blue along the paths in $G[V_i]$.
For each remaining directed edge $uv$ of $G$ we have $u \in V_i$ and $v \in V_j$ for some $i \neq j$.
Color $uv$ red if $i < j$ and blue if $i > j$.
See Figure~\ref{fig:trees-and-colorings}.
Assume that there is a red copy of $T_2$.
Since the out-degree of each vertex in $G$ is at most~$2$, each non-leaf vertex of $T_2$ has two incoming edges.
Thus the root $r$ is in $V_2$ and at least one of its neighbors is in $V_1$, a contradiction.\\
 
The trees $T_1$ and $T_2$ have $106$ and $21$ vertices, respectively, and are illustrated in Figure~\ref{fig:trees-and-colorings}.
We know that every planar avoidable tree has at least $8$ vertices since it has radius at least three and it is not a generalized broom. 

\begin{figure}
 \centering
 \includegraphics{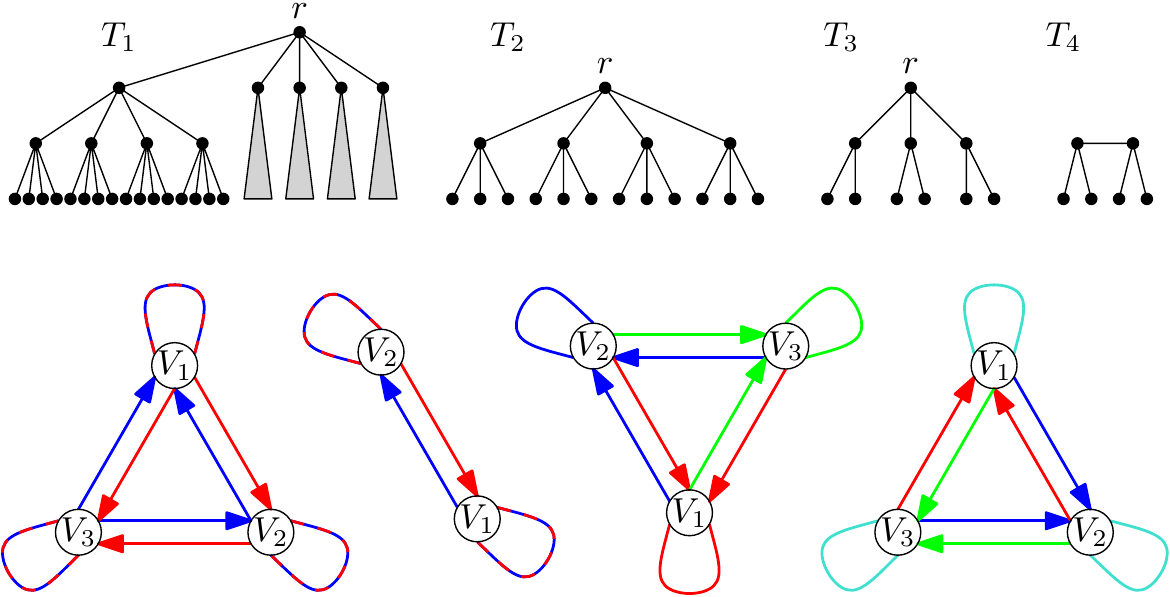}
 \caption{Illustrations of trees $T_1,\ldots,T_4$ defined in the proofs of Theorem~\ref{not-good} and~\ref{more-colors}:
 $T_1$ is planar avoidable with $2$ colors.
 $T_2$ is avoidable with $2$ colors in the class of outerplanar graphs.
 $T_3$ is planar avoidable with $3$ colors.
 $T_4$ is planar avoidable with $4$ colors.
 The colorings below illustrate patterns of how to color any planar (outerplanar) graph on basis of a partition $V_1,V_2,V_3$ ($V_1,V_2$) of the vertices, and an orientation of the edges between the parts.
 }
 \label{fig:trees-and-colorings}
\end{figure}

%%%%%%%%%%%%%%%%%%%%%%%%%%%%%%%%
%%%%%%%%%%%%%%%%%%%%%%%%%%%%%%%%
%%%%%%%%%%%%%%%%%%%%%%%%%%%%%%%%
\section{Proof of Theorem~\ref{more-colors}}\label{sec:more-colors}
%%%%%%%%%%%%%%%%%%%%%%%%%%%%%%%%
%%%%%%%%%%%%%%%%%%%%%%%%%%%%%%%%
%%%%%%%%%%%%%%%%%%%%%%%%%%%%%%%%

A result of Nash-Williams~\cite{NW} implies that any planar graph can be edge-decomposed into at most three forests.
Thus any graph $H$ that is not a forest is $3$-planar avoidable.
Another result of Gon\c{c}alves~\cite{Go1} states that any planar graph can be edge-colored in four colors so that each color class is a forest of caterpillars.
Thus any graph $H$ that is not a caterpillar forest is $4$-planar avoidable.\\

For the remainder of the proof  let $G$ be any planar graph.
Let $V_1,V_2,V_3$ be a partition of the vertex set $V(G)$ so that $G[V_i]$ is a linear forest~\cite{Poh}.
We shall define two colorings $c_3$ and $c_4$ of the edges of $G$ with three and four colors, respectively.
To this end, consider the bipartite subgraphs $B_1,B_2,B_3$ of $G$ with partitions $(V_2,V_3)$, $(V_1,V_3)$, $(V_1,V_2)$, and containing all edges of $G$ between respective parts.
For each $i=1,2,3$ orient the edges of $B_i$ so that the out-degree at each vertex is no more than~$2$.
(Such an orientation exists by~\cite{H,NW} as bipartite $n$-vertex planar graphs have no more than $2n-3$ edges, by Euler's formula.)

\begin{description}
 \item[Coloring $c_3$:]
  For $i=1,2,3$, color all edges in $G[V_i]$ and all edges of $G$ that are oriented incoming at a vertex of $V_i$ in color~$i$.
  
 \item[Coloring $c_4$:]
  For $i=1,2,3$, color all edges of $G$ that are oriented incoming at a vertex of $V_i$ in color~$i$.
  Further, color all edges in $G[V_1]$, $G[V_2]$, $G[V_3]$ in color~$4$.
\end{description}

\noindent
Next we show that a tree $T_3$ of radius~$2$ with root $r$ and all vertices of distance $0,1$ to $r$ of degree at least~$3$ (see Figure~\ref{fig:trees-and-colorings}) is $3$-planar avoidable.
We claim that $c_3$ does not contain a monochromatic copy of $T_3$.
In fact, if $v$ is any vertex with at least three incident edges of the same color $i$, then $v$ must be a vertex in $V_i$.
However, $G[V_i]$ has maximum degree at most~$2$, while the vertices of degree at least~$3$ in $T_3$ induce a subgraph of maximum degree at least~$3$.
Hence there is no monochromatic copy of $T_3$ in $G$ under coloring $c_3$.\\

\noindent
Finally, we show that a symmetric double star $T_4$ on $6$ vertices, i.e, a tree with two adjacent vertices of degree~$3$ and four leaves (see Figure~\ref{fig:trees-and-colorings}) is $4$-planar avoidable.
We claim that $c_4$ does not contain a monochromatic copy of $T_4$.
First, color~$4$ is a disjoint union of paths, and thus there is no copy of $T_4$ in color~$4$.
For color $i \in \{1,2,3\}$ we see that, as before, only vertices in $V_i$ may have three incident edges of color $i$.
However, as $V_i$ is an independent set in the subgraph of color $i$, there is no copy of $T_4$ in that subgraph.
Hence there is no monochromatic copy of $T_4$ in $G$ under coloring $c_4$, as desired. \qed \\

Let us remark that coloring $c_3$ shows that every graph $H$ in which the vertices of degree at least~$4$ induce a subgraph of maximum degree at least~$3$ is $3$-planar avoidable.
Similarly, coloring $c_4$ shows that every graph $H$ with an odd-length path whose two endvertices have degree at least three each, is $4$-planar avoidable.

\section{Conclusions}\label{conclusions}
In this paper we initiated the study of Ramsey properties of planar graphs.
When two colors are considered, only some outerplanar bipartite graphs are unavoidable and even some trees are avoidable.
We showed that $C_4$ is unavoidable. The following questions remain open:\\

{\bf 1.} Are other even cycles unavoidable? \\

{\bf 2.} What is the smallest number of vertices in an avoidable tree?\\

All of our positive results, showing that some graphs are unavoidable, use the fact that the iterated triangulation $\Tr(n)$ arrows these graphs.\\

{\bf 3.} Is is true that for each planar unavoidable graph $H$ there is $n=n(H)$ such that $\Tr(n) \rightarrow G$?

\end{document}